\documentclass[a4paper, 11pt, reqno]{amsart}

\usepackage[a4paper]{geometry}
\usepackage[utf8]{inputenc}
\usepackage[english]{babel}
\usepackage[T1]{fontenc}
\usepackage{textcomp}
\usepackage{lmodern}
\usepackage{graphicx}

\usepackage{amsaddr}

\usepackage{todonotes}

\usepackage[refpage]{nomencl}
\makenomenclature

\usepackage{amsmath}
\usepackage{amsfonts}
\usepackage{amsthm}
\usepackage{latexsym}
\usepackage{amssymb}
\usepackage[all,cmtip]{xy}

\DeclareMathOperator{\mor}{Mor}
\DeclareMathOperator{\dgmor}{\underline{Mor}}
\DeclareMathOperator{\Hom}{Hom}

\DeclareMathOperator{\Fun}{Fun}
\DeclareMathOperator{\Fundg}{Fun_{dg}}
\DeclareMathOperator{\Ob}{Ob}

\DeclareMathOperator{\Funex}{Fun_{ex}}
\DeclareMathOperator{\Funinf}{Fun_\infty}
\DeclareMathOperator{\RHom}{\mathbb{R}\underline{Hom}}

\DeclareMathOperator{\Perf}{Perf}
\DeclareMathOperator{\Perfdg}{Perf_{dg}}

\DeclareMathOperator{\car}{char}

\DeclareMathOperator{\qcoh}{QCoh}

\newcommand{\bbcat}{\mathbb}
\newcommand{\cat}{\mathbf}

\newcommand{\opp}[1]{{#1}^{\mathrm{op}}}
\newcommand{\kat}{\mathsf}

\newcommand{\lder}{\mathbb L}
\newcommand{\rder}{\mathbb R}
\newcommand{\lotimes}{\otimes^\lder}

\newcommand{\qis}{\overset{\mathrm{qis}}{\approx}}
\newcommand{\qe}{\overset{\mathrm{qe}}{\approx}}

\newcommand{\tria}[1]{\mathrm{tria}(#1)}
\newcommand{\pretr}[1]{\mathrm{pretr}(#1)}
\newcommand{\per}[1]{\mathrm{per}(#1)}
\newcommand{\perdg}[1]{\mathrm{per_{dg}}(#1)}

\newcommand{\basering}[1]{\mathbf{#1}}
\newcommand{\enrich}[1]{\underline{#1}}

\newcommand{\Mod}[1]{\mathsf{Mod}(#1)}

\newcommand{\comp}[1]{\mathsf{C}(#1)}
\newcommand{\hocomp}[1]{\mathsf{K}(#1)}
\newcommand{\dercomp}[1]{\mathsf{D}(#1)}
\newcommand{\compdg}[1]{\mathsf{C}_\mathrm{dg}(#1)}

\newcommand{\dercat}[1]{\mathfrak{D}(#1)}
\newcommand{\dercatdg}[1]{\mathfrak{D}_{\mathrm{dg}}(#1)}

\newcommand{\hproj}[1]{\mathrm{h\textrm{-}proj}(#1)}

\newcommand{\rqrep}[1]{\mathrm{qrep}^r(#1)}

\newcommand{\HOdir}[2]{\Phi^{#1 \to #2}}

\newtheorem{thm}{Theorem}[section]

\newtheorem{prop}[thm]{Proposition}
\newtheorem{coroll}[thm]{Corollary}

\newtheorem{lemma}[thm]{Lemma}

\theoremstyle{remark}
\newtheorem{remark}[thm]{Remark}
\newtheorem{example}[thm]{Example}

\theoremstyle{definition}
\newtheorem{defin}[thm]{Definition}

\numberwithin{equation}{section}

\entrymodifiers={+!!<0pt,\fontdimen22\textfont2>}   

\title{The uniqueness problem of dg-lifts and Fourier-Mukai kernels}
\author{Francesco Genovese}
\address{Dipartimento di Matematica ``F. Casorati'', Università di Pavia, \\ Via Ferrata 5, 27100 Pavia (PV), Italy}
\email{francesco.genovese01@ateneopv.it}


\begin{document}
\begin{abstract}
We address the uniqueness problem of dg-lifts of exact functors between triangulated categories, and its relationship with the uniqueness problem of Fourier-Mukai kernels. We prove a positive result under a vanishing hypothesis on the functors, employing $A_\infty$-categorical techniques.
\end{abstract}

\maketitle
\section{Introduction} \label{section:geometric_motivation}
\emph{Triangulated categories} are nowadays a classical topic in mathematics, with many applications in geometry and algebra. In particular, they arise in algebraic geometry as derived categories of (quasi-)coherent sheaves on schemes. Their serious technical drawbacks (in particular, the non functoriality of cones) suggest that they are actually ``shadows'' of more complicated, higher categorical structures. A popular way to enhance the understanding of triangulated categories is to employ \emph{differential graded (dg-) categories}, namely, categories enriched in complexes of modules over a ground field $\basering k$ (more in general, $\basering k$ can be taken as a commutative ring). A \emph{(dg-)enhancement} of a triangulated category $\cat T$ is a pretriangulated dg-category $\cat A$ such that $H^0(\cat A)$ is equivalent to $\cat T$; with the term \emph{pretriangulated dg-category} we mean a dg-category which, roughly speaking, contains shifts and functorial cones up to homotopy equivalence. If $\cat A$ is a pretriangulated dg-category, then its zeroth cohomology $H^0(\cat A)$ has a natural structure of triangulated category; a dg-functor $F \colon \cat A \to \cat B$ (which is simply a functor of enriched categories) induces an exact functor $H^0(F) \colon H^0(\cat A) \to H^0(\cat B)$.  Unfortunately, dg-functors do not retain the homotopical structure of dg-categories; so, we must consider more complicated -- homotopy relevant -- replacements, namely, \emph{quasi-functors}. They can be described concretely as \emph{right quasi-representable bimodules} (see Proposition \ref{prop:rep_quasifun}) or as \emph{$A_\infty$-functors} (see Proposition \ref{prop:RHom_Ainf}). 

Quasi-functors $\cat A \to \cat B$ form a dg-category (defined up to quasi-equivalence), which is denoted by $\RHom(\cat A, \cat B)$. They yield ordinary functors by taking cohomology, namely, there is a functor:
\begin{equation}
H^0 = \HOdir{\cat A}{\cat B} \colon H^0(\RHom(\cat A, \cat B)) \to \Fun(H^0(\cat A),H^0(\cat B)). \nomenclature{$\HOdir{\cat A}{\cat B}$}{The functor which maps a quasi-functor to its $H^0$}
\end{equation}
If $\cat A$ and $\cat B$ are pretriangulated, then $\HOdir{\cat A}{\cat B}$ is viewed as taking values in the category of exact functors $\Funex(H^0(\cat A), H^0(\cat B))$. \nomenclature{$\Funex(\cat T, \cat T')$}{The category of exact functors between triangulated categories}
By definition, a \emph{dg-lift} of an exact functor $\overline{F} \colon H^0(\cat A) \to H^0(\cat B)$ is a quasi-functor $F \colon \cat A \to \cat B$ such that $H^0(F) \cong \overline{F}$. The \emph{uniqueness problem of dg-lifts}, which is the main topic of the work, amounts to studying whether, given quasi-functors $F, G \colon \cat A \to \cat B$, $H^0(F) \cong H^0(G)$ implies that $F \cong G$. The relevance of this problem lies in the fact that, in the geometric cases, it is essentially equivalent to the \emph{uniqueness problem of Fourier-Mukai kernels}. Let us make this claim precise.

Let $X$ be a quasi-compact and quasi-separated scheme (over $\basering k$). We denote by $\dercat{\qcoh(X)}$ \nomenclature{$\dercat{\qcoh(X)}$}{The derived category of quasi-coherent sheaves on a scheme $X$} the derived category of quasi-coherent sheaves on $X$. The subcategory of compact objects of $\dercat{\qcoh(X)}$ coincides with the category of perfect complexes $\Perf(X)$ \nomenclature{$\Perf(X)$}{The category of perfect complexes of quasi-coherent sheaves on a scheme $X$}. Given two schemes $X$ and $Y$, there is a functor:
\begin{equation}
\Phi^{X \to Y}_- \colon \dercat{\qcoh(X \times Y)} \to \Funex(\Perf(X), \dercat{\qcoh(Y)}), \nomenclature{$\Phi^{X \to Y}_{\mathcal E} = \Phi_{\mathcal E}$}{The Fourier-Mukai functor with kernel $\mathcal E$}
\end{equation}
which maps a complex $\mathcal E \in \dercat{\qcoh(X \times Y)}$ to its \emph{Fourier-Mukai functor}
\begin{equation*}
\Phi^{X \to Y}_{\mathcal E} = \Phi_{\mathcal E} \colon \Perf(X) \to \dercat{\qcoh(Y)},
\end{equation*}
which is defined by:
\begin{equation*}
\Phi_{\mathcal E}(-) = \rder(p_2)_*(\mathcal E \lotimes p_1^*(-)),
\end{equation*}
where $p_1 \colon X \times Y \to X$ and $p_2 \colon X \times Y \to Y$ are the natural projections. If an exact functor $F \colon \Perf(X) \to \dercat{\qcoh(Y)}$ is such that $F \cong \Phi_{\mathcal E}$, we say thay $\mathcal E$ is a \emph{Fourier-Mukai kernel} of $F$. Current research is devoted to investigating the properties of $\Phi^{X \to Y}_-$ (see \cite{canonaco-survey} for a survey); for instance, the uniqueness problem of Fourier-Mukai kernels amounts to studying if $\Phi^{X \to Y}_{\mathcal E} \cong \Phi^{X \to Y}_{\mathcal E'}$ implies $\mathcal E \cong \mathcal E'$.  In general, we know that this is false: we can find counterexamples when $X=Y$ is an elliptic curve (see \cite{canonaco-nonuniquenessFM}). However, we do obtain a positive answer in some particular cases: for instance, it is known that fully faithful functors $F \colon \Perf(X) \to \dercat{\qcoh{Y}}$ admit uniquely determined Fourier-Mukai kernels, when $X$ and $Y$ are smooth projective (see \cite[Theorem 2.2]{orlov-equivalences-derived} for the original formulation).

Now, let us see how this is related to dg-categories and quasi-functors. If $X$ is a scheme (over $\basering k$), then the derived category $\dercat{\qcoh(X)}$ has an enhancement, which we call $\dercatdg{\qcoh(X)}$,\nomenclature{$\dercatdg{\qcoh(X)}$}{A chosen dg-enhancement of $\dercat{\qcoh(X)}$} choosing it once and for all and identifying $H^0(\dercatdg{\qcoh(X)}) = \dercat{\qcoh(X)}$. Taking the dg-subcategory of $\dercatdg{\qcoh(X)}$ whose objects correspond to $\Perf(X)$, we find an enhancement $\Perfdg(X)$ \nomenclature{$\Perfdg(X)$}{A chosen dg-enhancement of $\Perf(X)$} of the category of perfect complexes. A remarkable theorem by B. Toën tells us that, under suitable hypotheses, \emph{every quasi-functor has a unique Fourier-Mukai kernel}, in the following sense:
\begin{thm}[Adapted from {\cite[Theorem 8.9]{toen-morita}}] \label{thm:toen_FM}
Let $X$ and $Y$ be quasi-compact and separated schemes over $\basering k$. Then, there is an isomorphism in $\kat{Hqe}$:
\begin{equation} \label{eq:toen_FM}
\dercatdg{\qcoh(X \times Y)} \xrightarrow{\sim} \RHom(\Perfdg(X), \dercatdg{\qcoh(Y)}.
\end{equation}
\end{thm}

Next, a result adapted from \cite[Theorem 1.1]{lunts-schnurer-newenh} gives the desired ``bridge'' between Fourier-Mukai functors and quasi-functors between dg-categories:
\begin{thm} \label{thm:dglift_fouriermukai}
Let $X$ and $Y$ be Noetherian separatedproblem is unsolved, even if some partial results have been obtained. For instance, schemes over $\basering k$ such that $X \times Y$ is Noetherian and the following condition holds for both $X$ and $Y$: any perfect complex is isomorphic to a strictly perfect complex (i. e. a bounded complex of vector bundles). Then, there is a commutative diagram (up to isomorphism):
\begin{equation}
\begin{gathered}
\xymatrix{
\dercat{\qcoh(X \times Y)} \ar[r]^-{\sim} \ar[dr]_-{\Phi^{X \rightarrow Y}_{-}} & H^0(\RHom(\Perfdg(X), \dercatdg{\qcoh(Y)} ) \ar[d]^{\HOdir{\Perfdg(X)}{\dercatdg{\qcoh(Y)}}} \\
& \Funex(\Perf(X), \dercat{\qcoh{Y})},
}
\end{gathered}
\end{equation}
where the horizontal equivalence is induced by \eqref{eq:toen_FM}.
\end{thm}
\begin{remark}
The hypotheses of the above theorem are satisfied if both $X$ and $Y$ are quasi-projective.
\end{remark}
The above result tells us that, under suitable hypotheses, the properties of $\Phi^{X \to Y}_-$ are directly translated to those of $\HOdir{\Perfdg(X)}{\dercatdg{\qcoh(Y)}}$. In particular, the dg-lift uniqueness problem for functors $\Perf(X) \to \dercat{\qcoh(Y)}$ (with the above chosen dg-enhancements) is equivalent to the uniqueness problem of Fourier-Mukai kernels.

Next, we state the main theorem of the work, which is purely algebraic. Its proof employs the description of quasi-functors by means of $A_\infty$-functors; even if it involves some rather intricate computations with the $A_\infty$ formalism, it is not conceptually difficult.
\begin{thm}[Theorem \ref{thm:dglift_unique_with_vanishing}]
Let $\cat A$ and $\cat B$ be triangulated dg-categories. Assume that $\cat A$ is the \emph{triangulated hull} (see Subsection \ref{subsec:pretr_hull}) of a $\basering k$-linear category $\bbcat E$. Moreover, let $F, G \colon \cat A \to \cat B$ be quasi-functors satisfying the following vanishing condition:
\begin{equation*}
H^j(\cat B(F(E), F(E'))) \cong 0,
\end{equation*}
for all $j < 0$, for all $E, E' \in \bbcat E$. Then, $H^0(F_{|\mathbb E}) \cong H^0(G_{|\mathbb E})$ implies $F \cong G$.
\end{thm}
From this theorem, we obtain a result giving uniqueness of dg-lifts in the case where the source dg-category is an enhancement of the subcategory of compact objects in a suitable Verdier quotient of the derived category of a $\basering k$-linear category (see Theorem \ref{thm:dglift_unique_quotient}). This applies in particular to $\Perf(X)$, when $X$ is a quasi-projective scheme; from this we obtain the following geometric application: 
\begin{thm}[Theorem \ref{thm:cap5_uniqueness_geometricapplication}]
Let $X$ and $Y$ be schemes satisfying the hypotheses of Theorem \ref{thm:dglift_fouriermukai}, with $X$ quasi-projective. Let $\mathcal E, \mathcal E' \in \dercat{\qcoh(X \times Y)}$ be such that
\begin{equation*}
{\Phi^{X \to Y}_\mathcal{E}} \cong {\Phi^{X \to Y}_\mathcal{E'}} \cong F \colon \Perf(X) \to \dercat{\qcoh(Y)},
\end{equation*}
and $\Hom(F(\mathcal O_X(n)), F(\mathcal O_X(m))[j]) = 0$ for all $j < 0$, for all $n,m \in \mathbb Z$. Then $\mathcal E \cong \mathcal E'$.
\end{thm}
The above result is an improvement of \cite[Theorem 1.1]{canonaco-twistedFM}, clearly only regarding the uniqueness problem and the non-twisted case: our result holds not only for smooth projective varieties, and with a hypothesis which is weaker than the one in the mentioned article, which is:
\begin{equation*}
\Hom(F(\mathcal F), F(\mathcal G)[j]) = 0,
\end{equation*}
for all $\mathcal F, \mathcal G \in \mathrm{Coh}(X)$. It is also an improvement of \cite[Remark 5.7]{canonaco-supportedFM}, which holds for fully faithful functors. 

\subsection*{Acknowledgements}

This article has been estracted from the author's PhD thesis, written under the supervision of Prof. Alberto Canonaco. In particular, the author thanks him for giving key suggestions on the part about the geometric applications.
\section{Dg-categories and quasi-functors}
This section contains the basic well-known facts about dg-categories; we refer to \cite{keller-dgcat} for a comprehensive survey. We fix, once and for all, a ground field $\basering k$. Virtually every category we shall encounter will be at least $\basering k$-linear, so we allow ourself some sloppiness, and often employ the terms ``category'' and ``functor'' meaning ``$\basering k$-category'' and ``$\basering k$-functor''. 
\subsection{Dg-categories and dg-functors}
A dg-category is a category enriched over the closed symmetric monoidal category $\comp{\basering k}$ of cochain complexes of $\basering k$-modules:
\begin{defin} \label{def:cap1_dgcat}
A \emph{differential graded (dg-) category} $\cat A$ consists of a set of objects $\Ob \cat A$, a hom-complex $\cat A(A,B)$ for any couple of objects $A,B$, and (unital, associative) composition chain maps of complexes of $\basering k$-modules:
\begin{equation*}
\begin{split}
\cat A(B,C) \otimes \cat A(A,B) & \to \cat A(A,C), \\
g \otimes f & \mapsto gf = g \circ f
\end{split}
\end{equation*}
\end{defin}
\begin{defin} \label{def:cap1_dgfun}
Let $\cat A$ and $\cat B$ be dg-categories. A \emph{dg-functor} $F$ consists of the following data:
\begin{itemize}
\item a function $F \colon \Ob \cat A \to \Ob \cat B$;
\item for any couple of objects $(A,B)$ of $\cat A$, a chain map
\begin{equation*}
F = F_{(A,B)} \colon \cat A(A,B) \to \cat B(F(A),F(B)),
\end{equation*}
\end{itemize}
subject to the usual associativity and unitality axioms.
\end{defin}
\begin{example}
An example of dg-category is given by the \emph{dg-category of complexes} $\compdg{\basering k$}: it has the same objects as $\comp{\basering k}$, and complexes of morphisms $\enrich{\Hom}(V,W)$ given by:
\begin{equation} \label{eq:internalhom_complex_def}
\begin{split}
\enrich{\Hom}(V,W)^n &= \prod_{i \in \mathbb Z} \Hom(V^i, W^{i+k}), \\
df &= d_W \circ f - (-1)^{|f|} f \circ d_V.
\end{split}
\end{equation}
\end{example}
\begin{remark}
All usual categorical constructions can be carried out for dg-categories and dg-functors. 
\begin{enumerate}
\item Any ordinary ($\basering k$-linear) category can be viewed as a dg-category, with trivial complexes of morphisms.
\item For any dg-category $\cat A$ there is the \emph{opposite dg-category} $\opp{\cat A}$, such that
\begin{equation*}
\opp{\cat A}(A,B) = \cat A(B,A), 
\end{equation*}
with the same compositions as in $\cat A$ up to a sign:
\begin{equation*}
\opp{f} \opp{g} = (-1)^{|f||g|} \opp{(gf)},
\end{equation*}
denoting by $\opp{f} \in \opp{\cat A}(B,A)$ the corresponding morphism of $f \in \cat A(A,B)$.
\item Given dg-categories $\cat A$ and $\cat B$, there is the \emph{tensor product} $\cat A \otimes \cat B$: its objects are couples $(A,B)$ where $A \in \cat A$ and $B \in \cat B$; its hom-complexes are given by
\begin{equation*}
(\cat A \otimes \cat B) ((A,B),(A',B'))= \cat A(A,A') \otimes \cat B(B,B').
\end{equation*}
Compositions of two morphisms $f \otimes g$ and $f' \otimes g'$ is given by:
\begin{equation*}
(f' \otimes g')(f \otimes g) = (-1)^{|g'||f|} f'f \otimes g'g.
\end{equation*}
The tensor product commutes with taking opposites: $\opp{(\cat A \otimes \cat B)} = \opp{\cat A} \otimes \opp{\cat B}$. Also, it is symmetric, namely, there is an isomorphism of dg-categories: $\cat A \otimes \cat B \cong \cat B \otimes \cat A$.
\item Given dg-categories $\cat A$ and $\cat B$, there is a dg-category $\Fundg(\cat A, \cat B)$ whose objects are dg-functors $\cat A \to \cat B$ and whose complexes of morphisms are the so-called \emph{dg-natural transformations}:
A dg-natural transformation $\varphi \colon F \to G$ of degree $p$ is a collection of degree $p$ morphisms
\begin{equation*}
\varphi_A \colon F(A) \to G(A),
\end{equation*}
for all $A \in \cat A$, such that for any degree $q$ morphism $f \in \cat A(A,A')$ the following diagram is commutative up to the sign $(-1)^{|p||q|}$:
\begin{equation*}
\xymatrix{
F(A) \ar[r]^{\varphi_A} \ar[d]_{F(f)}  & G(A) \ar[d]^{G(f)} \\
F(A') \ar[r]^{\varphi_{A'}} & G(A').
}
\end{equation*}
Differentials and compositions of dg-natural transformations are defined objectwise. 

There is a natural isomorphism of dg-categories:
\begin{equation} \label{eq:sec1_iso_fundg_internalhom}
\Fundg(\cat A \otimes \cat B, \cat C) \cong \Fundg(\cat A, \Fundg(\cat B, \cat C)).
\end{equation}
Dg-functors $\cat A \otimes \cat B \to \cat C$ are called \emph{dg-bifunctors}, and they are ``dg-functors of two variables $A \in \cat A$ and $B \in \cat B$'', separately dg-functorial in both.

\item Given a dg-category $\cat A$, a \emph{right $\cat A$-dg-module} is a dg-functor $\opp{\cat A} \to \compdg{\basering k}$. We set
\begin{equation*}
\compdg{\cat A} = \Fundg(\opp{\cat A}, \compdg{\basering k}).
\end{equation*}
We have a fully faithful dg-functor
\begin{equation}
\begin{split}
h = h_{\cat A} \colon \cat A & \to \compdg{\cat A}, \\
A & \mapsto \cat A(-,A),
\end{split}
\end{equation}
which is the dg version of the Yoneda embedding.

Given dg-categories $\cat A$ and $\cat B$, an \emph{$\cat A$-$\cat B$-dg-bimodule} (covariant in $\cat A$, contravariant in $\cat B$) is a right $\cat B \otimes \opp{\cat A}$-dg-module, namely, a dg-functor $\opp{\cat B} \otimes \cat A \to \compdg{\basering k}$. By convention, the contravariant variable comes first. By \eqref{eq:sec1_iso_fundg_internalhom}, such a bimodule can also be viewed as a dg-functor $\cat A \to \compdg{\cat B}$.
\end{enumerate}
\end{remark}

\subsection{The derived category}

Small dg-categories and dg-functors form a category, which is denoted by $\kat{dgCat}_{\basering{k}}$, or simply $\kat{dgCat}$ \nomenclature{$\kat{dgCat}$}{The category of (small) dg-categories} when the base ring is clear. The operations of taking cocycles and cohomology can be extended from complexes of $\basering k$-modules to dg-categories and dg-functors:
\begin{defin}
Let $\cat A$ be a dg-category. The \emph{underlying category} (resp. the \emph{homotopy category}) of $\cat A$ is the category $Z^0(\cat A)$ (resp. $H^0(\cat A)$) which is defined as follows:
\begin{itemize}
\item $\Ob Z^0(\cat A) = \Ob H^0(\cat A) = \Ob \cat A$,
\item $Z^0(\cat A)(A,B) = Z^0(\cat A(A,B))$ (respectively $H^0(\cat A)(A,B) = H^0(\cat A(A,B))$), for all $A, B \in \cat A$,
\end{itemize}
with natural compositions and identities.
\end{defin}
The mappings $\cat A \mapsto Z^0(\cat A)$ and $\cat A \mapsto H^0(\cat A)$ are functorial: given a dg-functor $F \colon \cat A \to \cat B$, there are natural induced functors 
\begin{align*}
Z^0(F) \colon Z^0(\cat A) & \to Z^0(\cat B), \\
H^0(F) \colon H^0(\cat A) & \to H^0(\cat B). 
\end{align*}
Given two objects $A,B$ in a dg-category $\cat A$, we say that they are \emph{dg-isomorphic} (resp. \emph{homotopy equivalent}), and write $A \cong B$ (resp. $A \approx B$) if they are isomorphic in $Z^0(\cat A)$ (resp. $H^0(\cat A)$).

Let $\cat A$ be a dg-category. The \emph{homotopy category of $\cat A$-modules} is defined to be $\hocomp{\cat A} = H^0(\compdg{\cat A})$. A morphism $M \to N$ in $\hocomp{\cat A}$ is a \emph{quasi-isomorphism} if $M(A) \to N(A)$ is a quasi-isomorphism of complexes for all $A \in \cat A$. The \emph{derived category} of $\cat A$ is defined to be the localisation of $\hocomp{\cat A}$ along quasi-isomorphisms:
\begin{equation*}
\dercomp{\cat A} = \hocomp{\cat A}[\mathrm{Qis}^{-1}].
\end{equation*}
The Yoneda embedding induces a fully faithful functor:
\begin{equation} \label{eq:der_Yoneda}
H^0(\cat A) \hookrightarrow \dercomp{\cat A},
\end{equation}
which is called the \emph{derived Yoneda embedding}. The category $\dercomp{\cat A}$ is triangulated; as in the case of the derived category of complexes of $\basering k$-modules, morphisms $T \to T'$ in $\dercomp{\cat A}$ are represented by ``roofs''
\begin{equation*}
T \xleftarrow{\approx} T'' \rightarrow T',
\end{equation*}
in $\hocomp{\cat A}$, where the arrow $T'' \to T$ is a quasi-isomorphism. Two $\cat A$-dg-modules $T$ and $T'$ are \emph{quasi-isomorphic} ($T \qis T'$) if they are isomorphic in $\dercomp{\cat A}$, which is equivalent to saying that there is a ``roof'' of quasi-isomorphisms:
\begin{equation}
T \xleftarrow{\approx} T'' \xrightarrow{\approx} T'.
\end{equation}

We denote by $\tria{\cat A}$ the smallest strictly full triangulated subcategory of $\dercomp{\cat A}$ which contains the image of \eqref{eq:der_Yoneda}. Moreover, we denote by $\per{\cat A}$ the idempotent completion of $\tria{\cat A}$, which coincides with the smallest strictly full triangulated subcategory of $\dercomp{\cat A}$ which contains the image of \eqref{eq:der_Yoneda} and it is thick, i.e. closed under direct summands; it can also be characterised as the subcategory of compact objects in $\dercomp{\cat A}$. The derived Yoneda embedding factors through $\tria{\cat A}$:
\begin{equation}
H^0(\cat A) \hookrightarrow \tria{\cat A} \hookrightarrow \per{\cat A}.
\end{equation}
\begin{defin} \label{def:pretr_tr_dgcat}
A dg-category $\cat A$ is \emph{pretriangulated} if $H^0(\cat A) \hookrightarrow \tria{\cat A}$ is an equivalence; it is \emph{triangulated} if $H^0(\cat A) \hookrightarrow \per{\cat A}$ is an equivalence.
\end{defin}
We remark that a pretriangulated dg-category $\cat A$ is triangulated if and only if $H^0(\cat A)$ is idempotent complete. 

Pretriangulated dg-categories are employed as higher categorical models for triangulated categories. A \emph{dg-enhancement} of a triangulated category $\cat T$ is a pretriangulated dg-category $\cat A$ such that $H^0(\cat A)$ is equivalent to $\cat T$.

\subsection{Quasi-functors}
The category $\kat{dgCat}$ carries significant homotopical structure. A \emph{quasi-equivalence} is a dg-functor $F \colon \cat A \to \cat B$ such that the maps
\begin{equation*}
F_{(A,B)} \colon \cat A(A,B) \to \cat B(F(A),F(B))
\end{equation*}
are quasi-isomorphisms, and $H^0(F)$ is essentially surjective. Given dg-categories $\cat A$ and $\cat B$, we say that they are \emph{quasi-equivalent}, writing $\cat A \qe \cat B$, \nomenclature{$\cat A \qe \cat B$}{Dg-categories $\cat A, \cat B$ are quasi-equivalent} if there exists a zig-zag of quasi-equivalences:
\begin{equation*}
\cat A \leftarrow \cat A_1 \rightarrow \ldots \leftarrow \cat A_n \rightarrow \cat B.
\end{equation*}
Model category theory allows us to understand dg-categories up to quasi-equivalence. We summarise the main results in the following statement:
\begin{thm}[\cite{tabuada-dgcat}, \cite{toen-morita}]
The category $\kat{dgCat}$ of small dg-categories has a model category structure whose weak equivalences are the quasi-equivalences; the localisation of $\kat{dgCat}$ along quasi-equivalences is denoted by $\kat{Hqe}$. Given dg-categories $\cat A$ and $\cat B$, there exists a dg-category $\RHom(\cat A, \cat B)$ (which is defined up to isomorphism in $\kat{Hqe}$ and depends only on the quasi-equivalence classes of $\cat A$ and $\cat B$) such that there is a natural isomorphism in $\kat{Hqe}$:
\begin{equation}
\RHom(\cat A \otimes \cat B, \cat C) \cong \RHom(\cat A, \RHom(\cat B, \cat C)).
\end{equation}
\end{thm}
Objects of $\RHom(\cat A, \cat B)$ are called \emph{quasi-functors}: they are the ``homotopy relevant'' functors between dg-categories. Quasi-functors can be described concretely as particular bimodules:
\begin{prop}[{\cite[Theorem 4.5]{keller-dgcat}}] \label{prop:rep_quasifun}
The category $H^0(\RHom(\cat A, \cat B))$ is naturally equivalent to the category $\rqrep{\cat B \otimes \opp{\cat A}}$ of \emph{right quasi-representable} $\cat A$-$\cat B$-dg-bimodules, namely, the full subcategory of $\dercomp{\cat B \otimes \opp{\cat A}}$ of dg-functors $T \colon \cat A \to \compdg{\cat B}$ such that $T(A)$ is quasi-isomorphic to a $\cat B$-module of the form $\cat B(-,F(A))$, for some $F(A) \in \cat B$, for all $A \in \cat A$.
\end{prop}

\subsection{Pretriangulated hulls} \label{subsec:pretr_hull}
Let $\cat A$ be a dg-category. A dg-module $X \in \compdg{\cat A}$ is \emph{acyclic} if $X(A)$ is an acyclic complex for all $A \in \cat A$. A dg-module $M \in \compdg{\cat A}$ is \emph{h-projective} if $\hocomp{\cat A}(M,X) \cong 0$ for any acyclic $\cat A$-dg-module $X$. H-projective dg-modules serve as an enhancement of the derived category:
\begin{prop}
The full dg-subcategory $\hproj{\cat A} \subset \compdg{\cat A}$ of h-projective $\cat A$-dg-modules is an enhancement of $\dercomp{\cat A}$. Namely, the functor
\begin{equation*}
H^0(\hproj{\cat A}) \subset \hocomp{\cat A} \to \dercomp{\cat A}
\end{equation*}
is an equivalence. 
\end{prop}
Taking suitable dg-subcategories of $\hproj{\cat A}$, we obtain enhancements of $\tria{\cat A}$ and $\per{\cat A}$, respectively $\pretr{\cat A}$ and $\perdg{\cat A}$. For instance, $\perdg{\cat A}$ is by definition the full dg-subcategory of $\hproj{\cat A}$ whose objects are the same as $\per{\cat A}$. The dg-Yoneda embedding factors through $\pretr{\cat A}$:
\begin{equation}
\cat A \hookrightarrow \pretr{\cat A} \subset \perdg{\cat A}.
\end{equation}

The dg-category $\perdg{\cat A}$ is called the \emph{triangulated hull} of $\cat A$. It satisfies the following ``homotopy universal property'':
\begin{prop}[{\cite[Theorem 7.2]{toen-morita}, \cite[Corollary 4.2]{canonaco-stellari-internalhoms}}] \label{prop:triadgcat_uproperty}
Let $\cat A, \cat B$ be dg-categories, and assume that $\cat B$ is triangulated. Then $\RHom(\cat A, \cat B)$ is triangulated. Moreover, there is a natural quasi-equivalence:
\begin{equation} \label{eq:perdg_univproperty}
\RHom(\perdg{\cat A}, \cat B) \xrightarrow{\sim} \RHom(\cat A, \cat B),
\end{equation}
induced by the Yoneda embedding $\cat A \hookrightarrow \perdg{\cat A}$.
\end{prop}

\section{The dg-lift uniqueness problem}

Any quasi-functor $T \colon \cat A \to \cat B$ yields an ordinary functor $H^0(T) \colon H^0(\cat A) \to H^0(\cat B)$. More precisely, there is a functor: 
\begin{equation} \label{eq:H0functor_RHom}
\HOdir{\cat A}{\cat B} \colon H^0(\RHom(\cat A, \cat B)) \to \Fun(H^0(\cat A), H^0(\cat B)).
\end{equation}
When we identify $H^0(\RHom(\cat A, \cat B)$ with $\rqrep{\cat B \otimes \opp{\cat A}}$, $\HOdir{\cat A}{\cat B}$ is precisely the functor which maps a bimodule $T$ to its (objectwise) zeroth cohomology $H^0(T)$, which is a $H^0(\cat A)$-$H^0(\cat B)$-bimodule, namely, a functor
\begin{equation*}
H^0(T) \colon \opp{H^0(\cat B)} \otimes H^0(\cat A) \to \Mod{\basering k},
\end{equation*}
or equivalently a functor
\begin{equation*}
H^0(T) \colon H^0(\cat A) \to \Fun(\opp{H^0(\cat B)}, \Mod{\basering k}),
\end{equation*}
where $\Mod{\basering k}$ is the category of $\basering k$-modules. Since $T$ is right quasi-representable, then $H^0(T)$ is right representable, namely
\begin{equation*}
H^0(T)(A) \cong H^0(\cat B)(-,F(A)),
\end{equation*}
for some $F(A) \in \cat B$, for all $A \in \cat A$. Hence, $H^0(T)$ yields a functor $H^0(\cat A) \to H^0(\cat B)$.

If $\cat A$ and $\cat B$ are pretriangulated, then $\HOdir{\cat A}{\cat B}$ takes values in the category of exact functors $\Funex(H^0(\cat A), H^0(\cat B))$. The \emph{dg-lift uniqueness problem} amounts to understanding in which cases $\HOdir{\cat A}{\cat B}$ is essentially injective, that is: given quasi-functors $T_1, T_2$ such that $H^0(T_1) \cong H^0(T_2)$, is it true that $T_1 \cong T_2$ in $H^0(\RHom(\cat A, \cat B))$? In many situations, we will be studying dg-functors whose domain dg-category $\cat A$ is (pre)triangulated and generated by a simpler dg-category, namely, $\cat A \qe \perdg{\cat C}$. In this case, the dg-lift uniquness problem can be reduced to generators:
\begin{lemma} \label{lemma:dglift_reduction_generators}
Let $\cat A$ and $\cat B$ be triangulated dg-categories, and assume that $\cat A \qe \perdg{\cat C}$ for some dg-category $\cat C$. Then, $\HOdir{\cat A}{\cat B}$ is essentially injective if $\HOdir{\cat C}{\cat B}$ is such.
\end{lemma}
\begin{proof}
Without loss of generality, we may identify $\cat A = \perdg{\cat C}$. There is a commutative diagram:
\begin{equation*}
\xymatrix{
H^0(\RHom(\perdg{\cat C}, \cat B)) \ar[d]^\sim  \ar[r]^{\HOdir{\cat A}{\cat B}} & \Funex(H^0(\perdg{\cat C}), H^0(\cat B)) \ar[d] \\
H^0(\RHom(\cat C, \cat B)) \ar[r]^{\HOdir{\cat C}{\cat B}} & \Fun(H^0(\cat C), H^0(\cat B)),
}
\end{equation*}
where the left vertical arrow is induced by the Yoneda embedding $\cat C \hookrightarrow \perdg{\cat C}$, and the right vertical arrow is induced by its zeroth cohomology: $H^0(\cat C) \hookrightarrow H^0(\perdg{\cat C})$. By Proposition \ref{prop:triadgcat_uproperty}, the left vertical arrow is an isomorphism; the claim now follows from a direct argument.
\end{proof}
We mention another relevant property of $\HOdir{\cat A}{\cat B}$:
\begin{prop} \label{prop:H0_conservative}
The functor \eqref{eq:H0functor_RHom} reflects isomorphisms.
\end{prop}
\begin{proof}
A morphism $T \to T'$ in $H^0(\RHom(\cat A, \cat B)) = \rqrep{\cat A, \cat B}$ is given by a roof
\begin{equation*}
T \xleftarrow{\approx} T'' \rightarrow T'
\end{equation*}
in $\hocomp{\cat B \otimes \opp{\cat A}}$, where the arrow $T'' \to T$ is a quasi-isomorphism. So, it sufficient to prove that any morphism of $\cat A$-$\cat B$-bimodules $\varphi : T \to T'$ is a quasi-isomorphism if $H^0(\varphi) : H^0(T) \to H^0(T')$ is an isomorphism. Now, $\varphi$ is a quasi-isomorphism if and only if $\varphi_A : T(A) \to T'(A)$ is an isomorphism in $\dercomp{\cat B}$ for all $A \in \cat A$, which is equivalent to requiring that $\varphi'_A \colon \cat B(-,F(A)) \to \cat B(-,F'(A))$ is an isomorphism in $\dercomp{\cat B}$, where $\varphi'_A$ is the unique morphism in $\dercomp{\cat B}$ such that the following diagram is commutative in $\dercomp{\cat B}$:
\begin{equation*}
\xymatrix{
T(A) \ar[r]^{\varphi_A} \ar[d]^\approx & T'(A) \ar[d]^\approx \\
\cat B(-,F(A)) \ar[r]^{\varphi'_A} & \cat B(-, F'(A)).
}
\end{equation*}
Now, by the derived Yoneda embedding of $\cat B$, $\varphi'_A$ is a quasi-isomorphism if and only if $\varphi'_A(1_{F(A)}) \in \cat B(F(A), F'(A))$ is a homotopy equivalence. This means that $[\varphi'_A(1_{F(A)})] = H^0(\varphi'_A)([1_{F(A)}])$ is an isomorphism in $H^0(\cat B)$, so by the Yoneda embedding of the (ordinary) category $H^0(\cat B)$ this is equivalent to requiring that
\begin{align*}
H^0(\varphi'_A) \colon H^0(\cat B)(-, F(A)) \to H^0(\cat B)(-,F'(A))
\end{align*}
is an isomorphism in $\Fun(H^0(\cat B), \Mod{\basering k})$. Taking $H^0$, the above commutative diagram becomes:
\begin{equation*}
\xymatrix{
H^0(T)(A) \ar[r]^{H^0(\varphi)_A} \ar[d]^\sim & H^0(T')(A) \ar[d]^\sim \\
\cat H^0(B)(-,F(A)) \ar[r]^{H^0(\varphi'_A)} & \cat H^0(B)(-, F'(A)).
}
\end{equation*}
By hypothesis, $H^0(\varphi)_A$ is an isomorphism, so we deduce that $H^0(\varphi'_A)$ is an isomorphism for all $A \in \cat A$; by the above discussion, this implies that $\varphi_A \colon T(A) \to T'(A)$ is a quasi-isomorphism for all $A$, and we are done.
\end{proof}

\section{Dg-lifts and $A_\infty$-functors} \label{chapter:Ainfinity}
\emph{$A_\infty$-categories} and $A_\infty$-functors are, respectively, a homotopy coherent incarnation of dg-categories and dg-functors. $A_\infty$-functors are actually a model for quasi-functors; their advantage over quasi-representable bimodules relies in their ``concreteness'': they are defined by elementary (even if quite complicated) formulae, which can be employed in rather direct arguments. This formalism will allow us to prove a dg-lift uniqueness result under some hypothesis on the functors involved

\subsection{$A_\infty$-categories and functors}
The basic notions of the theory of $A_\infty$-categories and functors are taken directly from \cite{seidel-fukaya}, whose conventions will be followed. We warn the reader especially about sign conventions, which are possibly the most annoying feature of the theory. If it feels more comfortable, just assume that $\car \basering k = 2$, at least at a first reading.

We will be working with \emph{strictly unital} $A_\infty$-categories and functors. The formal definitions are as follows:
\begin{defin}
A \emph{strictly unital} $A_\infty$-category $\cat A$ consists of a set of objects $\Ob \cat A$, a graded $\basering k$-vector space $\cat A(X_0,X_1)$ for any couple of objects $X_0,X_1 \in \cat A$, and multilinear composition maps for any order $d \geq 1$:
\begin{equation}
\mu^d_{\cat A} \colon \cat A(X_{d-1}, X_d) \otimes \ldots \otimes \cat A(X_0, X_1) \to \cat A(X_0, X_d)[2-d],
\end{equation}
satisfying the following collection of equations (for all $d \geq 1$):
\begin{equation}
\sum_{m=1}^d \sum_{n=0}^{d-m} (-1)^{\maltese_n} \mu_{\cat A}^{d-m+1}(f_d,\ldots,f_{n+m+1}, \mu^m_{\cat A}(f_{n+m},\ldots,f_{n+1}),f_n,\ldots,f_1) = 0,
\end{equation}
where by definition $\maltese_n = |f_1| + \ldots + |f_n| - n$. Moreover, for any object $X \in \cat A$, there exists a (necessarily unique) morphism $1_X \in \cat A(X,X)^0$ which satisfies:
\begin{equation}
\begin{split}
& \mu^1_{\cat A}(1_X) = 0, \\
& (-1)^{|f|} \mu^2_{\cat A}(1_{X_1}, f) = \mu^2_{\cat A}(f, 1_{X_0}) = f, \quad \forall\,f \in \cat A(X_0,X_1), \\
& \mu^d_{\cat A}(f_{d-1}, \ldots,f_{n+1},1_{X_n},f_n,\ldots,f_1) = 0, \\
& \forall\, d>2, f_k \in \cat A(X_{k-1},X_k), \forall\, 0 \leq n < d. 
\end{split}
\end{equation}
\end{defin}
Unwinding the above definition, we find out that that the map $\mu^1_{\cat A}$ is a coboundary which endows the hom-spaces $\cat A(X, Y)$ with a structure of chain complex. The composition $\mu^2_{\cat A}$ is not associative, but its deviation from being so is measured by the higher order maps $\mu^d_{\cat A}$.
\begin{defin}
Let $\cat A$ and $\cat B$ be (strictly unital) $A_\infty$-categories. An $A_\infty$-functor $F \colon \cat A \to \cat B$ consists of a map of sets
\begin{align*}
F^0 \colon \Ob \cat A & \to \Ob \cat B, \\
X & \mapsto F^0(X) = F(X),
\end{align*}
and multilinear maps
\begin{equation}
F^d \colon \cat A(X_{d-1}, X_d) \otimes \ldots \otimes \cat A(X_0, X_1) \to \cat B(F(X_0),F(X_d))[1-d],
\end{equation}
subject to the following equations, for all $d \geq 1$:
\begin{equation} \label{eq:Ainf_equation}
\begin{split}
\sum_{r \geq 1} & \sum_{s_1 + \ldots + s_r = d} \mu^r_{\cat B}(F^{s_r}(f_d,\ldots,f_{d-s_r+1}),\ldots, F^{s_1}(f_{s_1},\ldots,f_1)) \\
&= \sum_{m=1}^d \sum_{n=0}^{d-m} (-1)^{\maltese_n} F^{d-m+1}(f_d,\ldots,f_{n+m+1},\mu^m_{\cat A}(f_{n+m},\ldots,f_{n+1}),f_n,\ldots,f_1),
\end{split}
\end{equation}
where $s_i \geq 1$ for all $i$. Moreover, $F$ is required to satisfy the following strict unitality condition:
\begin{equation} \label{eq:Ainf_functor_strictunit}
\begin{split}
& F^1(1_X) = 1_{F(X)}, \quad \forall\, X \in \cat A, \\
& F^d(f_{d-1},\ldots,f_{n+1},1_{X_n},f_n,\ldots, f_1) = 0, \\
& \forall\, d \geq 2, f_k \in \cat A(X_{k-1},X_k), \forall\, 0 \leq n < d. 
\end{split}
\end{equation}

Given $A_\infty$-functors $F \colon \cat A \to \cat B$ and $G \colon \cat B \to \cat C$, their \emph{composition} $G \circ F$ is defined as follows:
\begin{equation} \label{eq:Ainf_composition_functors}
\begin{split}
(G \circ F)^0  &= G^0 \circ F^0,  \\
(G \circ F)^d &(f_d,\ldots,f_1)  \\
 &= \sum_{r \geq 1} \sum_{s_1 + \ldots + s_r = d}  G^r(F^{s_r}(f_d,\ldots,f_{d-s_r+1}), \ldots, F^{s_1}(f_{s_1}, \ldots, f_1)),
\end{split}
\end{equation}
whenever $d \geq 1$, with $s_i \geq 1$.
\end{defin}
\begin{remark} \label{remark:A_infty_funct-reduced}
Any dg-category $\cat A$ can be viewed as an $A_\infty$-category, setting
\begin{align*}
\mu^1_{\cat A}(f) &= (-1)^{|f|} df, \\
\mu^2_{\cat A}(g,f) &= (-1)^{|f|} gf, \\
\mu^d_{\cat A} &= 0, \quad \forall\, d > 2.
\end{align*}
As we see, apart from sign twists, a dg-category is an $A_\infty$-category whose higher compositions (for $d>2$) vanish. From now on, unless otherwise specified, any dg-category will be implicitly viewed in this way as an $A_\infty$-category.

It is interesting to see how the definition of $A_\infty$-functor behaves if the domain and codomain are assumed to be dg-categories. If $F \colon \cat A \to \cat B$ is an $A_\infty$-functor between dg-categories, the degree $d$ equation \eqref{eq:Ainf_equation} boils down to:
\begin{equation} \label{eq:A_infty_funct-dg}
\begin{split}
\mu^1_{\cat B}(F^d &(f_d, \ldots, f_1))  + \sum_{j=1}^{d-1} \mu^2_{\cat B}(F^j(f_d,\ldots,f_{d-j+1}), F^{d-j}(f_{d-j}, \ldots, f_1)) \\
= \sum_{n=0}^{d-1} & (-1)^{\maltese_n} F^d(f_d, \ldots, f_{n+2},\mu^1_{\cat A}(f_{n+1}),f_n,\ldots, f_1) \\
&+ \sum_{n=0}^{d-2}  (-1)^{\maltese_n} F^{d-1}(f_d,\ldots, f_{n+3}, \mu^2_{\cat A}(f_{n+2}, f_{n+1}), f_n, \ldots , f_1).
\end{split}
\end{equation}
In the even simpler case when $F \colon \bbcat E \to \cat B$ is an $A_\infty$-functor where $\bbcat E$ is a $\basering k$-linear category and $\cat B$ is a dg-category, the degree $d$ equation defining $F$ then reduces to the following:
\begin{equation} \label{eq:A_infty_funct-reduced}
\begin{split}
\mu^1_{\cat B}(F^d(f_d, \ldots, f_1)) & + \sum_{j=1}^{d-1} \mu^2_{\cat B}(F^j(f_d,\ldots, f_{d-j+1}), F^{d-j}(f_{d-j}, \ldots, f_1))\\
 = \sum_{n=0}^{d-2} & (-1)^{\maltese_n} F^{d-1}(f_d, \ldots f_{n+3}, \mu^2_{\bbcat E}(f_{n+2}, f_{n+1}),f_n, \ldots,f_1).
\end{split}
\end{equation}

It is also interesting to see what is the composition of an $A_\infty$-functor $F \colon \cat A \to \cat B$ (between dg-categories) with a dg-functor $G \colon \cat B \to \cat C$. Such a dg-functor, viewed as an $A_\infty$-functor, is characterised by having $G^d = 0$ for all $d>1$. Formula \eqref{eq:Ainf_composition_functors} becomes very simple:
\begin{equation}
(G \circ F)^d(f_d,\ldots,f_1) = G^1(F^d(f_d,\ldots,f_1)),
\end{equation}
for all $d \geq 1$. 
\end{remark}
Given $A_\infty$-categories $\cat A$ and $\cat B$, there is an \emph{$A_\infty$-category $\Funinf(\cat A, \cat B)$ \nomenclature{$\Funinf(\cat A, \cat B)$}{The $A_\infty$-category of strictly unital $A_\infty$-functors between two given $A_\infty$-categories} of (strictly unital) $A_\infty$-functors}. Its definition involves describing \emph{($A_\infty$-)natural transformations} of $A_\infty$-functors.
\begin{defin}
Let $F, G \colon \cat A \to \cat B$ be $A_\infty$-functors. A \emph{degree $g$ pre-natural transformation} $h \colon F \to G$ is consists of a sequence of maps $(h^0, h^1,\ldots)$ such that
\begin{equation*}
h^0 \colon X \mapsto h^0_X \in \cat B(F(X),G(X))^g, \quad X \in \cat A,
\end{equation*}
and $h^d$ is a family of multilinear maps
\begin{equation*}
h^d \colon \cat A(X_{d-1}, X_d) \otimes \ldots \otimes \cat A(X_0, X_1) \to \cat B(F(X_0),G(X_d))[g-d]
\end{equation*}
for any family of objects $X_0,\ldots,X_d \in \cat A$. Pre-natural transformations $F \to G$ form the graded vector space $\Funinf(\cat A,\cat B)(F,G)$. Compositions are described in \cite[Paragraph (1d)]{seidel-fukaya}. For example, we have that
\begin{equation*}
\mu^1(h)^0_X =\mu^1_{\cat B}(h^0_X), \quad \forall\, X \in \cat A. 
\end{equation*}
Moreover, we require the strict unitality condition:
\begin{equation} \label{eq:Ainf_nattrans_strictunit}
h^d(f_{d-1},\ldots,f_{n+1}, 1_{X_n}, f_n,\ldots, f_1) = 0,
\end{equation}
for all $d \geq 1$ and $0 \leq n < d$, with $f_k \in \cat A(X_{k-1}, X_k)$.
\end{defin} 
\begin{remark}
It is worth writing down the coboundary formula for a pre-natural transformation $h \colon F \to G$ when $F, G \colon \cat A \to \cat B$ are $A_\infty$-functors between dg-categories. If $d \geq 1$, we have:
\begin{equation} \label{eq:coboundary_nattrans}
\mu^1(h)^d (f_d, \ldots, f_1) = A^d - B^d,
\end{equation}
where
\begin{equation} \label{eq:Ad_nattrans}
\begin{split}
A^d = & \mu^1_{\cat B}(h^d(f_d,\ldots,f_1)) \\
&+ \mu^2_{\cat B}(G^d(f_d,\ldots,f_1), h_{X_0}) + (-1)^{\maltese_d(|h|-1)} \mu^2_{\cat B}(h_{X_d}, F^d(f_d,\ldots,f_1)) \\
& + \sum_{j=1}^{d-1} \mu^2_{\cat B}(G^j(f_d, \ldots, f_{d-j+1}), h^{d-j}(f_{d-j},\ldots,f_1)) \\
&+ \sum_{j=1}^{d-1} (-1)^{\maltese_{d-j}(|h|-1)} \mu^2_{\cat B}(h^j(f_d,\ldots,f_{d-j+1}), F^{d-k}(f_{d-j},\ldots, f_1)),
\end{split}
\end{equation}
and
\begin{equation} \label{eq:Bd_nattrans}
\begin{split}
B^d =& \sum_{n=0}^{d-1} (-1)^{\maltese_n +|h|-1} h^d(f_d, \ldots, f_{n+2},\mu^1_{\cat A}(f_{n+1}),f_n,\ldots, f_1) \\
&+ \sum_{n=0}^{d-2}  (-1)^{\maltese_n + |h|-1} h^{d-1}(f_d,\ldots, f_{n+3}, \mu^2_{\cat A}(f_{n+2}, f_{n+1}), f_n, \ldots , f_1),
\end{split}
\end{equation}
given composable morphisms $f_1,\ldots,f_d$ with first source $X_0$ and final target $X_d$. Notice that the term $B_d$ is similar to the right hand side of \eqref{eq:A_infty_funct-dg}.
\end{remark}

\subsection{Natural transformations}
Closed degree $0$ pre-natural transformations of $A_\infty$-functors are properly called \emph{natural transformations}. We are going to describe a useful characterisation of them; we start with a definition in the dg-setting:
\begin{defin} \label{defin:cap1_dgmor}
Let $\cat A$ be a dg-category (here, \emph{not} viewed as an $A_\infty$-category). The \emph{dg-category of (homotopy coherent) morphisms} $\dgmor \cat A$ \nomenclature{$\dgmor \cat A$}{The dg-category of homotopy coherent morphisms in a dg-category $\cat A$} is defined as follows. Objects are triples $(A,B,f)$, where $f \in Z^0(\cat A(A,B))$. A degree $n$ morphism $(A,B,f) \to (A',B',f')$ is given by a lower triangular matrix
\begin{equation*}
(u,v,h) = \begin{pmatrix}
u & 0 \\ h & v
\end{pmatrix},
\end{equation*}
where $u \in \cat A(A,A')^n$, $v \in \cat A(B,B')^n$ and $h \in \cat A(A,B)^{n-1}$. Compositions are defined by matrix multiplication with a sign rule:
\begin{equation*}
\begin{pmatrix}
u' & 0 \\ h' & v'
\end{pmatrix}
\begin{pmatrix}
u & 0 \\ h & v
\end{pmatrix} = 
\begin{pmatrix}
u'u & 0 \\ (-1)^n h'u+v'h & v'v
\end{pmatrix},
\end{equation*}
whenever $(u,v,h)$ has degree $n$. The differential of a morphism $(u,v,h) \colon (A,B,f) \to (A',B',f')$ of degree $n$ is defined by
\begin{equation*}
d \begin{pmatrix}
u & 0 \\ h & v
\end{pmatrix} = \begin{pmatrix}
du & 0 \\ dh + (-1)^n(f'u - vf) & dv
\end{pmatrix}.
\end{equation*}
\end{defin}
There are obvious ``source'' and ``target'' dg-functors:
\begin{align*}
& S \colon \dgmor \cat A \to \cat A, \quad (A,B,f) \mapsto A, \quad (u,v,h) \mapsto u, \\
& T \colon \dgmor \cat A \to \cat A, \quad (A,B,f) \mapsto B, \quad (u,v,h) \mapsto v.
\end{align*}
Notice that the chosen sign conventions in the definition of $\dgmor \cat A$ allow to define $S$ and $T$ in the simplest way, without any sign twist. 

We remark that there is a natural functor
\begin{equation} \label{eq:H0func_dgmor}
\begin{split}
H^0(\dgmor \cat A) & \to \mor H^0(\cat A), \\
(A,B,f) & \mapsto (A,B,[f]), \\
[(u,v,h)] & \mapsto ([u], [v]),
\end{split}
\end{equation}
where $\mor H^0(\cat A)$ \nomenclature{$\mor \cat C$}{The ordinary category of morphisms of $\cat C$} denotes the ordinary category of morphisms of $H^0(\cat A)$.
\begin{example} \label{example:dgmorA_Ainf}
Let $\cat A$ be a dg-category, now viewed as an $A_\infty$-category. Let us write down what happens when we view the dg-category of homotopy coherent morphisms $\cat Q =\dgmor \cat A$ as an $A_\infty$-category. First:
\begin{align*}
\mu^1_{\cat Q} & \begin{pmatrix}
u & 0 \\ h & v
\end{pmatrix} = (-1)^{|u|} \begin{pmatrix}
du & 0 \\ dh + (-1)^{|u|}(f'u-vf) & dv
\end{pmatrix} \\
&= (-1)^{|u|} \begin{pmatrix}
(-1)^{|u|} \mu^1_{\cat A}(u) & 0 \\
(-1)^{|u|-1} \mu^1_{\cat A}(h) + (-1)^{|u|}((-1)^{|u|}\mu^2_{\cat A}(f',u) - \mu^2_{\cat A} (v,f)) & (-1)^{|u|} \mu^1_{\cat A}(v)
\end{pmatrix} \\
&= \begin{pmatrix}
\mu^1_{\cat A}(u) & 0 \\
- \mu^1_{\cat A}(h) + (-1)^{|u|} \mu^2_{\cat A}(f',u) - \mu^2_{\cat A}(v,f) & \mu^1_{\cat A}(v)
\end{pmatrix}.
\end{align*}
Moreover:
\begin{align*}
\mu^2_{\cat Q} & \left( \begin{pmatrix}
u' & 0 \\ h' & v'
\end{pmatrix}, \begin{pmatrix}
u & 0 \\ h & v
\end{pmatrix} \right) \\
&= \begin{pmatrix}(-1)^{|u|} u'u & 0 \\ (-1)^{|u|}((-1)^{|u|} h'u + v'h) & (-1)^{|u|}v'v \end{pmatrix}\\
&= \begin{pmatrix}
\mu^2_{\cat A}(u',u) & 0 \\
(-1)^{|u|}\mu^2_{\cat A}(h',u) - \mu^2_{\cat A}(v',h) & \mu^2_{\cat A}(v',v)
\end{pmatrix}.
\end{align*}
\end{example}
Natural transformations of $A_\infty$-functors can now be characterised as ``directed homotopies'', in the sense explained by the following lemma.
\begin{lemma} \label{lemma:transform_hot}
Let $\cat A, \cat B$ be dg-categories. Let $F, G \colon \cat A \to \cat B$ be $A_\infty$-functors. There is a bijection between the set of (closed, degree $0$) natural transformations $F \to G$ and the set of $A_\infty$-functors $\varphi \colon \cat A \to \dgmor \cat B$ such that $S\varphi = F$ and $T \varphi = G$:
\begin{equation}
\varphi^d = (F^d, G^d, h^d) \leftrightarrow h^d.
\end{equation} 
\end{lemma}
\begin{proof}
Let $\varphi \colon \cat A \to \dgmor \cat B$ an $A_\infty$-functor as in the hypothesis. In particular, for any string of composable maps $f_1, \ldots, f_d$ with first source $X_0$ and final target $X_d$, we have
\begin{equation*}
\varphi^d(f_d, \ldots, f_1) = (F^d(f_d,\ldots, f_1), G^d(f_d,\ldots, f_1), h^d(f_d,\ldots,f_1))
\end{equation*}
as a morphism $(F(X_0),G(X_0),h_{X_0}) \to (F(X_d),G(X_d),h_{X_d})$. Notice that $F^d(\ldots)$ and $G^d(\ldots)$ have degree $|f_1| + \ldots + |f_d| +1 - d$, that is, $\maltese_d + 1$, whereas $h^d(\ldots)$ has degree $\maltese_d$. Now, we unwind the equation \eqref{eq:A_infty_funct-dg} which defines $\varphi$. By Example \ref{example:dgmorA_Ainf}, we have
\begin{equation*}
\mu^1(\varphi^d) = (\mu^1_{\cat B}(F^d), \mu^1_{\cat B}(G^d), -\mu^1_{\cat B}(h^d) + (-1)^{\maltese_d +1} \mu^2_{\cat B}(h_{X_d}, F^d) - \mu^2_{\cat B}(G^d, h_{X_0})).
\end{equation*}
Moreover:
\begin{align*}
\mu^2 &(\varphi^j(f_d,\ldots,f_{d-j+1}), \varphi^{d-j}(f_{d-j},\ldots,f_1)) \\
&= \mu^2((F^j,G^j,h^j),(F^{d-j}, G^{d-j}, h^{d-j})) \\
&=(\mu^2_{\cat B}(F^j,F^{d-j}), \mu^2_{\cat B}(G^j, G^{d-j}), (-1)^{\maltese_{d-j} +1}\mu^2_{\cat B}(h^j, F^{d-j}) - \mu^2_{\cat B}(G^j,h^{d-j})).
\end{align*}
Now, we find out that the left hand side of \eqref{eq:A_infty_funct-dg}, projected to the third component, is equal to the following:
\begin{align*}
-\mu^1_{\cat B}(h^d &(f_d,\ldots,f_1)) - \mu^2_{\cat B}(G^d(f_d,\ldots,f_1), h_{X_0}) - (-1)^{\maltese_d} \mu^2_{\cat B}(h_{X_d}, F^d(f_d,\ldots,f_1)) \\
- &\sum_{j=1}^{d-1} \mu^2_{\cat B}(G^j(f_d,\ldots,f_{d-j+1}),h^{d-j}(f_{d-j},\ldots,f_1)) \\
- & \sum_{j=1}^{d-1} (-1)^{\maltese_{d-j}} \mu^2_{\cat B}(h^j(f_d,\ldots,f_{d-j+1}), F^{d-j}(f_{d-j}, \ldots, f_1)).
\end{align*}
We immediately notice that the above term is equal to $-A^d$ when $|h|=0$ (see \eqref{eq:Ad_nattrans}). Moreover, the right hand side of \eqref{eq:A_infty_funct-dg}, projected to the third component, is equal to $-B^d$ when $|h|=0$ (see \eqref{eq:Bd_nattrans}). Now, it is clear that any $A_\infty$-functor $\varphi \colon \cat A \to \dgmor \cat B$ such that $S\varphi = F$ and $T \varphi = G$ defines a closed degree $0$ natural transformation $h \colon F \to G$, taking the projection of $\varphi$ to the third component; conversely, given $h \colon F \to G$ closed and of degree $0$, setting
\begin{equation*}
\varphi^d = (F^d, G^d, h^d)
\end{equation*}
we obtain an $A_\infty$-functor with the desired properties. Clearly, these mappings are mutually inverse. Moreover, the scrict unitality condition \eqref{eq:Ainf_functor_strictunit} for $\varphi$ is clearly equivalent to the strict unitality condition \eqref{eq:Ainf_nattrans_strictunit} for $h$.
\end{proof}
If $\cat A$ and $\cat B$ are dg-categories, then so is $\Funinf(\cat A, \cat B)$. Actually, this is an incarnation of $\RHom(\cat A, \cat B)$, as mentioned in \cite[Paragraph 4.3]{keller-dgcat}:
\begin{prop} \label{prop:RHom_Ainf}
The dg-category $\RHom(\cat A, \cat B)$ can be identified with the dg- category $\Funinf(\cat A, \cat B)$ of strictly unital $A_\infty$-functors from $\cat A$ to $\cat B$.
\end{prop}
The functor $\HOdir{\cat A}{\cat B}$ has clearly an incarnation in this setting:
\begin{equation}
\begin{split}
\HOdir{\cat A}{\cat B}  \colon & H^0(\Funinf(\cat A, \cat B))  \to \Fun(H^0(\cat A),H^0(\cat B)), \\
& F \mapsto H^0(F), \quad H^0(F)(f) = [F^1(f)], \\
& [h]_{\mu^1} \mapsto H^0(h), \quad H^0(h)_X = [h^0_X],
\end{split}
\end{equation}
where here $[\cdot]_{\mu^1}$ denotes the cohomology class with respect of the coboundary $\mu^1$ of $\Funinf(\cat A, \cat B)(F,G)$. Recalling Lemma \ref{lemma:transform_hot}, the action of the above functor on morphisms can also be viewed in terms of directed homotopies. Given $\varphi \colon \cat A \to \dgmor \cat B$ such that $S\varphi = F$ and $T\varphi =G$, we may identify $H^0(\varphi)$ to the ordinary functor
\begin{equation*}
H^0(\cat A) \to \mor(H^0(\cat B))
\end{equation*}
obtained by the following composition:
\begin{equation*}
H^0(\cat A) \to H^0(\dgmor \cat B) \xrightarrow{\eqref{eq:H0func_dgmor}} \mor(H^0(\cat B)).
\end{equation*}
\subsection{Uniqueness of dg-lifts}
The goal of this section is to prove a dg-lift uniqueness result using the formalism and techniques of $A_\infty$-functors. We will need the following (simplified) obstruction theory result, which can be proved with a direct computation. The analogue (general) result is proved for $A_\infty$-algebras in \cite[Corollaire B.1.5]{lefevre-Ainf}.
\begin{lemma} \label{lemma:cap5_obstructiontheory}
Let $\bbcat E$ be a $\basering k$-linear category, let $\cat B$ be a dg-category, and let $n \geq 2$ be an integer. Suppose that we have a finite sequence $(F^0, F^1, \ldots, F^{n-1})$, where $F^0 \colon \Ob \bbcat E \to \Ob \cat B, X \mapsto F(X) = F^0(X)$ and 
\begin{equation*}
F^d \colon \bbcat E(X_{d-1}, X_d) \otimes \ldots \otimes \bbcat E(X_0, X_1) \to \cat B(F(X_0),F(X_d))[1-d],
\end{equation*}
is a multilinear map, for all $d=1,\ldots,n-1$. Assume that \eqref{eq:A_infty_funct-reduced} is satisfied for all $d=1,\ldots,n-1$. Then, the expression
\begin{align*}
\sum_{j=0}^{n-2} & (-1)^{\maltese_j} F^{n-1}(f_n, \ldots f_{j+3}, \mu^2_{\bbcat E}(f_{j+2}, f_{j+1}),f_j, \ldots,f_1)  \\
 &- \sum_{j=1}^{n-1} \mu^2_{\cat B}(F^j(f_n,\ldots, f_{n-j+1}), F^{n-j}(f_{n-j}, \ldots, f_1))
\end{align*}
is a $\mu^1_{\cat B}$-cocycle, for any chain of composable maps $f_1,\ldots,f_n$.
\end{lemma}

Another key tool in our argument is the following lemma, which we first prove in the dg-framework, and then reinterpret with the $A_\infty$ notations:
\begin{lemma} \label{lemma:keylemma_lifting}
Let $\cat A$ be a dg-category, here \emph{not} viewed as an $A_\infty$-category. Let $(A,B,f)$ and $(A',B',f')$ be objects of $\dgmor \cat A$, and let $n \in \mathbb Z$ such that
\begin{equation*}
H^{n-1}(\cat A(A,B')) \cong 0.
\end{equation*}
Next, assume we are given a closed degree $n$ morphism $(u,v,h) \colon (A,B,f) \to (A',B',f')$. Then, if $u=d\tilde{u}$ and $v=d\tilde{v}$, there exists $\tilde{h} \colon A \to B'$ such that
\begin{equation*}
(u,v,h) = d(\tilde{u}, \tilde{v}, \tilde{h}).
\end{equation*}
\end{lemma}
\begin{proof}
By hypothesis we have $d(u,v,h) = 0$, in particular
\begin{equation*}
dh + (-1)^n(f'u - vf) = 0.
\end{equation*}
Now, $f'u = d(f'\tilde{u})$ and $vf = d(\tilde{v} f)$, and so
\begin{equation*}
d(h + (-1)^n(f'\tilde{u} - \tilde{v}f)) = 0
\end{equation*}
In other words, $h + (-1)^n(f'\tilde{u} - \tilde{v}f)$ is a $(n-1)$-cocycle. Hence, by hypothesis, it is a $(n-1)$-coboundary:
\begin{equation*}
h + (-1)^n(f'\tilde{u} - \tilde{v}f) = d\tilde{h}.
\end{equation*}
Finally, we compute:
\begin{equation*}
d \begin{pmatrix}
\tilde{u} & 0 \\ \tilde{h} & \tilde{v} 
\end{pmatrix} = \begin{pmatrix}
u & 0 \\ h + (-1)^n(f'\tilde{u} - \tilde{v} f) + (-1)^{n-1}(f'\tilde{u} - \tilde{v} f) & v
\end{pmatrix} = \begin{pmatrix}
u & 0 \\ h & v
\end{pmatrix}. \qedhere
\end{equation*}
\end{proof}
\begin{lemma} \label{lemma:keylemma_Ainf}
Let $\cat A$ be a dg-category, now viewed as an $A_\infty$-category. Let $(A,B,f)$ and $(A',B',f')$ be objects in $\cat Q = \dgmor \cat A$ (viewed as an $A_\infty$-category), and let $n \in \mathbb Z$ such that
\begin{equation*}
H^{n-1}(\cat A(A,B')) = 0.
\end{equation*}
Next, assume that we are given a degree $n$ morphism $(u,v,h) \colon (A,B,f) \to (A',B',f')$ such that $\mu^1_{\cat Q}(u,v,h)=0$. Then, if $u = \mu^1_{\cat A}(\tilde{u})$ and $v= \mu^1_{\cat A}(\tilde{v})$, there exists $\tilde{h} \colon A \to B'$ such that
\begin{equation*}
(u,v,h) = \mu^1_{\cat Q}(\tilde{u}, \tilde{v}, \tilde{h}).
\end{equation*}
\end{lemma}
\begin{proof}
Recall Example \ref{example:dgmorA_Ainf}. $u= \mu^1_{\cat A}(\tilde{u})$ means $(-1)^{n-1} u = d\tilde{u}$, and $(-1)^{n-1}v = d\tilde{v}$. Apply Lemma \ref{lemma:keylemma_lifting} to $(-1)^{n-1}(u,v,h)$:
\begin{equation*}
(-1)^{n-1}(u,v,h) = d(\tilde{u}, \tilde{v}, \tilde{h}) = (-1)^{n-1} \mu^1_{\cat Q}(\tilde{u}, \tilde{v}, \tilde{h}),
\end{equation*}
and the claim follows.
\end{proof}

We are going to prove the following claim, which is actually a lifting result of natural transformations:
\begin{prop} \label{prop:lift_transformation}
Let $\bbcat E$ be a $\basering k$-linear category, viewed as a dg-category concentrated in degree $0$, and let $\cat B$ be a dg-category. Let $F, G \colon \bbcat E \to \cat B$ be quasi-functors, such that
\begin{equation} \label{eq:condition2}
H^j(\cat B(F(E), G(E'))) = 0,
\end{equation}
for all $j < 0$ and for all $E, E' \in \bbcat E$. Let $\overline{\varphi}\colon H^0(F) \to H^0(G)$ be a natural transformation. Then, there exists a morphism $\varphi \colon F \to G$ in $H^0(\RHom(\bbcat E, \cat B))$ such that $H^0(\varphi) = \overline{\varphi}$.
\end{prop}
We obtain the following theorem, which is the announced dg-lift uniqueness result:
\begin{thm} \label{thm:dglift_unique_with_vanishing}
Let $\bbcat E$ be a $\basering k$-linear category, viewed as a dg-category concentrated in degree $0$, and let $\cat B$ be a triangulated dg-category. Let $F, G \colon \bbcat E \to \cat B$ be quasi-functors, such that
\begin{equation} \label{eq:condition}
H^j(\cat B(F(E), F(E'))) \cong 0,
\end{equation}
for all $j < 0$, for all $E, E' \in \bbcat E$. Let $\overline{\varphi}\colon H^0(F) \to H^0(G)$ be a natural isomorphism. Then, there exists an  isomorphism $\varphi \colon F \to G$ in $H^0(\RHom(\bbcat E, \cat B))$ such that $H^0(\varphi) = \overline{\varphi}$. 

In particular, set $\cat A = \perdg{\bbcat E}$, and view $\mathbb E$ as a full dg-subcategory of $\cat A$; if $F, G \colon \cat A \to \cat B$ are quasi-functors satisfying \eqref{eq:condition}, then $H^0(F_{|\mathbb E}) \cong H^0(G_{|\mathbb E})$ implies $F \cong G$ in $H^0(\RHom(\cat A, \cat B))$.
\end{thm}
\begin{proof}
Since $H^0(F) \cong H^0(G)$ and $\cat B$ is triangulated, then \eqref{eq:condition2} holds. Then, the proof is a direct application of Proposition \ref{prop:lift_transformation}, Proposition \ref{prop:H0_conservative}. The second part of the statement follows from Lemma \ref{lemma:dglift_reduction_generators}.
\end{proof}

Upon identifying $\RHom(\bbcat E, \cat B)$ to $\Funinf(\bbcat E, \cat B)$, Proposition \ref{prop:lift_transformation} is translated to the following:
\begin{prop}
Let $\bbcat E$ be a $\basering k$-linear category, viewed as a dg-category concentrated in degree $0$, and let $\cat B$ be a dg-category. Let $F, G \colon \bbcat E \to \cat B$ be (strictly unital) $A_\infty$-functors, satisfying
\begin{equation} \label{eq:condition2-Ainf}
H^j(\cat B(F^0(E), G^0(E'))) = 0,
\end{equation}
for all $j < 0$, for all $E, E' \in \bbcat E$. Assume $\overline{\varphi}\colon H^0(F) \to H^0(G)$ is a natural transformation. Then, there exists an $A_\infty$-natural transformation $\varphi \colon F \to G$, such that $H^0(\varphi) = \overline{\varphi}$.
\end{prop}
\begin{proof}
In view of Lemma \ref{lemma:transform_hot}, we try to define recursively a $A_\infty$-functor $\varphi \colon \bbcat E \to \dgmor \cat B$ such that $S \varphi = F, T \varphi = G$, and the induced functor 
\begin{equation*}
\bbcat E = H^0(\bbcat E) \to \mor (H^0(\cat B))
\end{equation*}
is equal to $\overline{\varphi}$. First, we define a map $\varphi^0$ on objects: for any $E \in \bbcat E$, we set
\begin{equation*}
\varphi^0(E) = (F^0(E),G^0(E), \varphi_E),
\end{equation*}
where $\varphi_E$ is a chosen lift of the given map $\overline{\varphi}_E \colon F^0(E) \to G^0(E)$. Next, we define $\varphi^1$ on a given basis (including the identities of all objects) of the space of morphisms. Given an element $f\colon E_0 \to E_1$ of this basis, we set
\begin{equation*}
\varphi^1(f) = (F^1(f), G^1(f), h^1(f)),
\end{equation*}
where $h^1(f)$ is a chosen degree $-1$ morphism such that
\begin{equation*}
-\mu^1_{\cat B}(h^1(f)) = \mu^2_{\cat B}(G^1(f), \varphi_{E_0}) - \mu^2_{\cat B} (\varphi_{E_1}, F^1(f)).
\end{equation*}
$h^1(f)$ exists by the hypothesis that  $\overline{\varphi} \colon H^0(F) \to H^0(G)$ is a natural transformation. Moreover, we may choose $h^1(1_E) = 0$ for all $E \in \bbcat E$. By construction, $\varphi^1(f)$ is a closed degree $0$ morphism in $\cat Q = \dgmor \cat B$ (see Example \ref{example:dgmorA_Ainf}), as required by \eqref{eq:A_infty_funct-reduced}, and $\varphi^1(1_E) = 1_{\varphi^0(E)}$.

Now, for $d \geq 2$, assume that we have defined a sequence of maps $(\varphi^1, \ldots, \varphi^{d-1})$ satisfying \eqref{eq:A_infty_funct-reduced} and strict unitality, with 
\begin{equation*}
\varphi^k(f_k, \ldots, f_1) = (F^k(f_k, \ldots, f_1), G^k(f_k, \ldots, f_1), h^k(f_k,\ldots, f_1)).
\end{equation*}
Given maps $f_i \colon E_{i-1} \to E_i$ in our chosen basis for $i=1,\ldots,d$, by Lemma \ref{lemma:cap5_obstructiontheory} the expression
\begin{equation} \label{eq:inductivestep}
\begin{split}
\sum_{n=0}^{d-2} & (-1)^{\maltese_n} \varphi^{d-1}(f_d, \ldots f_{n+3}, \mu^2_{\bbcat E}(f_{n+2}, f_{n+1}),f_n, \ldots,f_1)  \\
 &- \sum_{j=1}^{d-1} \mu^2_{\cat Q}(\varphi^j(f_d,\ldots, f_{d-j+1}), \varphi^{d-j}(f_{d-j}, \ldots, f_1))
\end{split}
\end{equation}
is a $\mu^1_{\cat Q}$-cocycle $(F^0(E_0), G^0(E_0), \varphi_{E_0}) \to (F^0(E_d),G^0(E_d), \varphi_{E_d})$, of degree $1-(d-1) = 2-d$. Since $F$ and $G$ are $A_\infty$-functors, we have that
\begin{equation*}
\eqref{eq:inductivestep} = (\mu^1_{\cat B}(F^d(f_d,\ldots,f_1), \mu^1_{\cat B}(G^d(f_d,\ldots,f_1), \cdots).
\end{equation*}
Then, the condition \eqref{eq:condition2-Ainf} allows us to apply Lemma \ref{lemma:keylemma_Ainf} (with $n=2-d$). We may choose $h^d(f_d, \ldots f_1)$ such that
\begin{equation*}
\eqref{eq:inductivestep} = \mu^1_{\cat Q}(F^d(f_d,\ldots,f_1), G^d(f_d,\ldots,f_1), h^d(f_d,\ldots, f_1)).
\end{equation*}
So, defining
\begin{equation*}
\varphi^d(f_d, \ldots, f_1) = (F^d(f_d,\ldots,f_1), G^d(f_d,\ldots,f_1), h^d(f_d,\ldots, f_1))
\end{equation*} 
we get the correct identity \eqref{eq:A_infty_funct-reduced}. Notice that, if one of the $f_i$ is an identity morphism, then expression \eqref{eq:inductivestep} vanishes, so in that case we may choose $h^d(f_d,\ldots,f_1)=0$, and hence $\varphi^d(f_d,\ldots, f_1) = 0$, which is the strict unitality condition. Finally, our result follows by recursion. 
\end{proof}
\section{Applications}
In this section we describe an application of the above technique which gives uniqueness results of Fourier-Mukai kernels. The dg-categories of interest in these applications are enhancements of Verdier quotients of the form $\dercomp{\bbcat A} / L$, where $\bbcat A$ is a $\basering k$-linear category and $L$ is a full subcategory of $\dercomp{\bbcat A}$ with suitable hypotheses. More precisely, we will work in the framework of the following result, whose proof is essentially contained in \cite[Section 6, first part]{orlov-dgenh}.
\begin{lemma} \label{lemma:localising_subcat_dercat}
Let $\bbcat A$ be a $\basering k$-linear category, viewed as a dg-category. Let $L \subseteq \dercomp{\bbcat A}$ be a localising subcategory (namely, strictly full and closed under direct sums), generated by compact objects $L^c = L \cap \dercomp{\bbcat A}^c$. \nomenclature{$\cat T^c$}{The subcategory of compact objects of a triangulated category $\cat T$} There is a canonical functor
\begin{equation} \label{eq:canonicalfunctor_localising_application}
\iota \colon \bbcat A \hookrightarrow \dercomp{\bbcat A}^c \to \dercomp{\bbcat A}^c / L^c \hookrightarrow (\dercomp{\bbcat A}/L)^c,
\end{equation}
where the composition of the last two maps is the restriction of the quotient functor $\dercomp{\bbcat A} \to \dercomp{\bbcat A}/L$. The triangulated category $(\dercomp{\bbcat A}/L)^c$ is the idempotent completion of $\dercomp{\bbcat A}^c / L^c$, and it is classically generated by the full subcategory with objects $\iota(\bbcat A)$. Moreover, if $\cat D$ together with the equivalence
\begin{equation*}
\epsilon \colon (\dercomp{\bbcat A}/L)^c \to H^0(\cat D)
\end{equation*}
is an enhancement of $(\dercomp{\bbcat A}/L)^c$, then $\cat D$ is quasi-equivalent to $\perdg{\bbcat A'}$, where $\bbcat A'$ is the full dg-subcategory of $\cat D$ whose object are given by $\epsilon(\iota(\bbcat A))$.
\end{lemma}
Verdier quotients such as $\dercomp{\bbcat A}^c / L^c$ are enhanced by the \emph{Drinfeld dg-quotient}. We state its definition and main properties, which we will need in the following; they are directly taken from \cite[1.6.2]{drinfeld-dgcat}.
\begin{defin}
Let $\cat A$ be a dg-category, and let $\cat B$ be a full dg-subcategory of $\cat A$. A \emph{dg-quotient} of $\cat A$ modulo $\cat B$ is a dg-category $\cat A / \cat B$ together with a quasi-functor $\pi \colon \cat A \to \cat A / \cat B$, such that for any dg-category $\cat C$ the induced functor
\begin{equation} \label{eq:dgquot_universalprop}
\pi^* \colon H^0(\RHom(\cat A / \cat B, \cat C)) \to H^0(\RHom(\cat A, \cat C))
\end{equation}
is fully faithful, and its essential image consists of quasi-functors $F \colon \cat A \to \cat C$ such that $H^0(F)$ maps objects of $\cat B$ to zero objects in $H^0(\cat C)$.
\end{defin}
\begin{thm}
Let $\cat A$ be a dg-category, and let $\cat B$ be a full dg-subcategory of $\cat A$. A dg-quotient $(\cat A / \cat B, \pi)$ exists, and it is uniquely determined up to natural quasi-equivalence. Moreover, if $\cat A$ is pretriangulated and $H^0(\cat B)$ is a triangulated subcategory of $H^0(\cat A)$, then $(H^0(\cat A / \cat B), H^0(\pi))$ is a Verdier quotient of $H^0(\cat A)$ modulo $H^0(\cat B)$:
\begin{equation}
H^0(\cat A)/H^0(\cat B) \xrightarrow{\sim} H^0(\cat A / \cat B).
\end{equation}
\end{thm}
\begin{remark} \label{remark:keyremark_dglift_uniqueness}
Assume the framework of Lemma \ref{lemma:localising_subcat_dercat}. We know that the category $\dercomp{\bbcat A}^c$ has $\perdg{\bbcat A}$ as a dg-enhancement. Moreover, taking $\mathcal L^c$ to be the full dg-subcategory of $\perdg{\bbcat A}$ whose objects correspond to $L^c$, we find out that the dg-quotient $\perdg{\bbcat A} / \mathcal L^c$ is an enhancement of $\dercomp{\bbcat A}^c/L^c$. Moreover, since $(\dercomp{\bbcat A}/L)^c$ can be viewed as the idempotent completion of $\dercomp{\bbcat A}^c/L^c$, we find out that the dg-category
\begin{equation*}
\perdg{\perdg{\bbcat A} / \mathcal L^c}
\end{equation*}
is an enhancement of $(\dercomp{\bbcat A}/L)^c$. Without loss of generality, we may assume that the above functor $\iota$ is obtained in $H^0$ by the quasi-functor
\begin{equation}
\tilde{\iota} \colon \bbcat A \hookrightarrow \perdg{\bbcat A} \xrightarrow{\pi} \perdg{\bbcat A}/\mathcal L^c \hookrightarrow \perdg{\perdg{\bbcat A} / \mathcal L^c}.
\end{equation}
Notice that the quasi-functor $\perdg{\bbcat A} \xrightarrow{\pi} \perdg{\bbcat A}/\mathcal L^c$ is the canonical projection to the dg-quotient, and the fully faithful dg-functor $\perdg{\bbcat A}/\mathcal L^c \hookrightarrow \perdg{\perdg{\bbcat A} / \mathcal L^c}$ is the canonical inclusion. They are respectively involved with the universal properties \eqref{eq:dgquot_universalprop} and \eqref{eq:perdg_univproperty}.

Now, \cite[Theorem 2.8]{orlov-dgenh} tells us that, under the vanishing hypothesis
\begin{equation} \label{eq:i(A)_hypothesis}
(\dercomp{\bbcat A}/L)(\iota(A),\iota(A')[j]) \cong 0, \quad \forall\, j < 0, \quad \forall\, A,A' \in \bbcat A,
\end{equation}
the category $(\dercomp{\bbcat A}/L)^c$ admits a \emph{unique} dg-enhancement (up to quasi-equivalence). So, in that case, we are allowed to identify any such enhancement $\cat D$, up to quasi-equivalence, to the dg-category $\perdg{\perdg{\bbcat A} / \mathcal L^c}$.
\end{remark}
Now, the abstract result of the previous section allows us to prove the following:
\begin{thm} \label{thm:dglift_unique_quotient}
Assume the framework of Lemma \ref{lemma:localising_subcat_dercat}, and assume that $(\dercomp{\bbcat A}/L)^c$ has a unique enhancement. Let $\cat D$ be such an enhancement, and for simplicity identify $H^0(\cat D) = (\dercomp{\bbcat A}/L)^c$. Let $F, G \colon \cat D \to \cat B$ be quasi-functors taking values in a triangulated dg-category $\cat B$, satisfying the vanishing hypothesis:
\begin{equation} \label{eq:vanishing_i(A)}
H^0(\cat B)(F(\iota(A)), F(\iota(A'))[j]) \cong 0, \quad \forall\, j < 0,
\end{equation}
for all $A, A' \in \bbcat A$. Then, if 
\begin{equation*}
H^0(F) \circ \iota \cong H^0(G) \circ \iota \colon \bbcat A \to \cat B,
\end{equation*}
we have that $F \cong G$ as quasi-functors.
\end{thm}
\begin{proof}
Recalling Remark \ref{remark:keyremark_dglift_uniqueness}, we are allowed to identify $\cat D$ to $\perdg{\perdg{\bbcat A} / \mathcal L^c}$. By the universal property of $\perdg{\perdg{\bbcat A} / \mathcal L^c}$, we have that $F \cong G$ if and only if $F_{|\perdg{\bbcat A}/\mathcal L^c} \cong G_{|\perdg{\bbcat A}/\mathcal L^c}$. Then, by the universal property of the dg-quotient, this is equivalent to 
\begin{equation*}
F_{|\perdg{\bbcat A}/\mathcal L^c} \circ \pi \cong G_{|\perdg{\bbcat A}/\mathcal L^c} \circ \pi \colon \perdg{\bbcat A} \to \cat B.
\end{equation*}
Finally, by the universal property of $\perdg{\bbcat A}$, this is equivalent to
\begin{equation*}
F \circ \tilde{\iota} \cong G \circ \tilde{\iota} \colon \bbcat A \to \cat B.
\end{equation*}
Now, recalling that we have identified $\iota = H^0(\tilde{\iota})$, a direct application of Theorem \ref{thm:dglift_unique_with_vanishing} gives the desired result.
\end{proof}

The above result has an interesting application. Let $X$ be a quasi-projective scheme, viewed as open subscheme of a projective scheme $\overline{X}$. Then, the derived category $\dercat{\qcoh(X)}$ of quasi-coherent sheaves on $X$ can be described as a quotient $\dercomp{\bbcat A} / L$. Namely, take $\bbcat A$ as the category with objects given by the integers, and
\begin{equation}
\bbcat A(i,j) = H^0(\overline{X},\mathcal O_{\overline{X}}(j-i)),
\end{equation}
with composition induced by that of the graded algebra $\bigoplus_n H^0(\overline{X},\mathcal O_{\overline{X}}(n))$. The subcategory $L$ is taken to be the category of all objects in $\dercomp{\bbcat A}$ whose cohomologies are ``$I$-torsion modules'' (for details, see \cite[before Corollary 7.8]{orlov-dgenh}). It can be proved that there is an equivalence $\dercomp{\bbcat A} /L \cong \dercat{\qcoh(X)}$, and also that the natural functor
\begin{equation*}
\bbcat A \hookrightarrow \dercomp{\bbcat A} \to \dercomp{\bbcat A} / L \xrightarrow{\sim} \dercat{\qcoh(X)}
\end{equation*}
maps any integer $j \in \Ob \bbcat A$ to the sheaf $\mathcal O_X(j)$. Moreover, the subcategory $L$ satisfies the hypotheses of Lemma \ref{lemma:localising_subcat_dercat}, and in particular the above discussion restricts to compact objects and perfect complexes. Namely, we have an equivalence $(\dercomp{\bbcat A} /L)^c \cong \Perf(X)$ such that, composed with the functor \eqref{eq:canonicalfunctor_localising_application}, gives:
\begin{equation}
\begin{split}
& \bbcat A \xrightarrow{\iota} (\dercomp{\bbcat A}/L)^c \xrightarrow{\sim} \Perf(X), \\
& j \mapsto \mathcal O_X(j).
\end{split}
\end{equation}
Now, let $\dercatdg{\qcoh(X)}$ be an enhancement of $\dercat{\qcoh(X)}$, and for simplicity identify this category to $H^0(\dercatdg{\qcoh(X)})$. Recall that the full dg-subcategory $\Perfdg(X)$ of $\dercatdg{\qcoh(X)}$ whose objects are the compact objects in $\dercat{\qcoh(X)}$ is an enhancement of $\Perf(X)$; also, recall that these enhancements are uniquely determined, by \cite[Corollary 7.8, Theorem 7.9]{orlov-dgenh}. Upon identifying $(\dercomp{\bbcat A}/L)^c$ to $\Perf(X)$ via the equivalence discussed above, we immediately get the following:
\begin{coroll}
Let $X$ be a quasi-projective scheme, and let $\cat B$ be a triangulated dg-category. Let $F, G \colon \Perfdg(X) \to \cat B$ be quasi-functors which satisfy the vanishing condition
\begin{equation*}
H^0(\cat B)(F(\mathcal O_X(n)), F(\mathcal O_X(m))[j])=0, \quad \forall\, j < 0,
\end{equation*}
for all $n,m \in \mathbb Z$. Then, if $H^0(F) \cong H^0(G)$, we have that $F \cong G$ as quasi-functors.
\end{coroll}

Finally, we apply this machinery to the uniqueness problem of Fourier-Mukai kernels, as explained in Section \ref{section:geometric_motivation}, hence obtaining the following uniqueness result:
\begin{thm} \label{thm:cap5_uniqueness_geometricapplication}
Let $X$ and $Y$ be schemes satisfying the hypotheses of Theorem \ref{thm:dglift_fouriermukai}, with $X$ quasi-projective. Let $\mathcal E, \mathcal E' \in \dercat{\qcoh(X \times Y)}$ be such that
\begin{equation*}
{\Phi^{X \to Y}_\mathcal{E}} \cong {\Phi^{X \to Y}_\mathcal{E'}} \cong F \colon \Perf(X) \to \dercat{\qcoh(Y)},
\end{equation*}
and $\Hom(F(\mathcal O_X(n)), F(\mathcal O_X(m))[j]) = 0$ for all $j < 0$, for all $n,m \in \mathbb Z$. Then $\mathcal E \cong \mathcal E'$.
\end{thm}
\bibliographystyle{alpha}
\bibliography{biblio}
\end{document}